\documentclass[12pt,twoside]{amsart}
\usepackage[all]{xy}   
\usepackage {amssymb,latexsym,amsthm,amsmath}
        \usepackage{hyperref}
        \usepackage[utf8]{inputenc}
        \usepackage{xcolor}
        
\long\def\symbolfootnote[#1]#2{\begingroup%
\def\thefootnote{\fnsymbol{footnote}}\footnote[#1]{#2}\endgroup}

\newcommand{\tr}{\ensuremath{{}^t\!}}

\newcommand{\C}{\mathbb C}

\newcommand{\Sp}{\mathrm{Sp} }

\newcommand{\diag}{\textup{diag}}
\newcommand{\R}{\mathbb R}

\def\ch{{\mathbf H}_{\mathbb C}}
\def \Sp {\mathrm{Sp}}
\def \PSp {\mathrm{PSp}}
\def \diag {\mathrm{diag}}
\def \V {\mathrm{V}}
\def \U {\mathrm{U}}
\def \W {\mathrm{W}}
\def \H {\mathbb{H}}

\makeatletter
\def\imod#1{\allowbreak\mkern10mu({\operator@font mod}\,\,#1)}
\makeatother

\newtheorem{theorem}{Theorem}[section]
\newtheorem{lemma}[theorem]{Lemma}
\newtheorem{corollary}[theorem]{Corollary}
\newtheorem{proposition}[theorem]{Proposition}
\newtheorem*{theorem*}{Theorem}
\theoremstyle{definition}
\newtheorem{remark}[theorem]{Remark}
\newtheorem{definition}[theorem]{Definition}

\numberwithin{equation}{section}

\newcommand{\ignore}[1]{}

\newcommand{\mynote}[1]{}
\newcommand{\secref}[1]{Section~\ref{#1}}
\newcommand{\thmref}[1]{Theorem~\ref{#1}}
\newcommand{\lemref}[1]{Lemma~\ref{#1}}

\newcommand{\propref}[1]{Proposition~\ref{#1}}
\newcommand{\corref}[1]{Corollary~\ref{#1}}
\newcommand{\eqnref}[1]{~{\textrm(\ref{#1})}}

\begin{document}
\setcounter{section}{0}
\setcounter{tocdepth}{1}
\title[Reversible Quaternionic Hyperbolic Isometries]{Reversible Quaternionic Hyperbolic Isometries}
\author[Sushil Bhunia]{Sushil Bhunia}
\email{sushilbhunia@gmail.com}
\thanks{Bhunia is supported by a SERB-NPDF (No. PDF/2017/001049) during the course of this work.}
\author{Krishnendu Gongopadhyay}
\email{krishnendu@iisermohali.ac.in, krishnendug@gmail.com}
\thanks{Gongopadhyay acknowledges partial support from SERB-DST MATRICS project:  MTR/2017/000355.  }
\address{Indian Institute of Science Education and Research (IISER) Mohali,
 Knowledge City,  Sector 81, S.A.S. Nagar 140306, Punjab, India}

 \subjclass[2010]{Primary 51M10; Secondary 15B33, 20E45. }
\keywords{ hyperbolic space, quaternions, reversible elements,  real elements, involutions, strongly reversible, strongly real. }
\date{\today}
\begin{abstract}
	Let $G$ be a group. An element $g\in G$ is called \textit{reversible} if it is conjugate to $g^{-1}$ within $G$, and called \textit{strongly reversible} if it is conjugate to $g^{-1}$ by an order two element of $G$. Let $\mathbf{H}^n_{\mathbb{H}}$ be the $n$-dimensional quaternionic hyperbolic space. Let $\PSp(n, 1)$ be the isometry group of $\mathbf{H}^n_{\mathbb{H}}$. In this paper, we classify reversible and strongly reversible elements in $\Sp(n)$ and $\Sp(n,1)$. Also,  we prove that all the elements of $\PSp(n, 1)$ are strongly reversible.
\end{abstract}
\maketitle

\section{Introduction}
Let $G$ be a group. An element $g\in G$ is called \textit{reversible} if $g^{-1}=xgx^{-1}$ for some $x\in G$. The terminology `real' has also been used extensively in the literature to refer the reversible elements, for example, see~\cite{El1},~\cite{El2},~\cite{EFN},~\cite{FZ},~\cite{KN1}, and 
~\cite{KN2},~\cite{ST1},~\cite{ST2},~\cite{TZ},~\cite{Wo}. A non-trivial element $g\in G$ is called an \textit{involution} if $g^2=1$. An element $g\in G$ is called \textit{strongly reversible} if $g^{-1}=xgx^{-1}$ for some $x\in G$ with $x^2=1$. Clearly, a strongly reversible element is reversible. An element $g$ is strongly reversible if and only if $g$ can be written as a product of two involutions. Every element of a conjugacy class which contains a reversible (resp. strongly reversible) element is reversible (resp. strongly reversible), i.e., reversibility (resp. strongly reversibility) is a conjugacy invariant.

Reversible and strongly reversible  group elements have been studied in many contexts. In classical group theory, a theorem of Frobenius and Schur states that  if $G$  is finite, the number of real-valued complex irreducible characters of $G$ equals the number of reversible conjugacy classes of $G$.  This has influenced a lot of research on reversibility from an algebraic point of view. However, the origin of research on `reversibility' can be traced back to works in classical dynamics and classical geometry. We refer the reader to the book by O'Farrel and Short \cite{OS} for an extensive account of reversibility in geometry and dynamics. 

The motivation of the present work comes from the investigations related to the reversibility in classical  geometries. 
Let ${\rm PO}(n,1)$ denote the full isometry group of the $n$-dimensional real hyperbolic space ${\mathbf H}^n_\R$ and let ${\rm PO}_o(n,1)$ denote the identity component, which is the group of orientation preserving isometries of ${\mathbf H}^n_\R$.  It is well-known that every element of ${\rm PO}(n,1)$ is strongly reversible, e.g.  \cite[Theorem 6.11]{OS}. The reversible elements in ${\rm PO}_o(n,1)$  have been classified in \cite{Go1}, and in \cite{sh1} using a different approach, and also see~\cite{LOS}. We refer to \cite[Chapter 6]{OS} for an extensive treatment of reversibility in Euclidean, spherical, and real hyperbolic geometries. It follows from these works that an element $g$ in ${\rm PO}_o(n,1)$ is reversible if and only if it is strongly reversible. 

The reversibility of isometries of the complex hyperbolic space has been investigated by Gongopadhyay and Parker in \cite{GP}. The group ${\rm U}(n,1)$ and ${\rm SU}(n,1)$ act as the holomorphic isometries of the $n$-dimensional complex hyperbolic space $\ch^n$. Reversible elements in these groups were classified in \cite{GP}. It follows from this work that an element $g$ in ${\rm U}(n,1)$ is reversible if and only if it is strongly reversible. An element $g$ in the full isometry group ${\rm PU}(n,1)$ is reversible (resp. strongly reversible) if and only if a certain lift of $g$ in ${\rm U}(n,1)$ is reversible (resp. strongly reversible), for more details, see \cite[Theorem 4.5]{GP}. 

The investigation of strong reversibility in ${\rm PO}_o(n,1)$ and ${\rm PU}(n,1)$ is related to the broader problem of finding the involution length in the respective groups. 
The relation between strong reversibility and involution length in any group $G$ is that: every element of $G$ is strongly reversible if and only if the involution length of $G$ is $2$. 
The involution length of ${\rm PO}_o(n,1)$ is $2$ or $3$ depending explicitly on the congruence class of $n\mod 4$, which is obtained by Basmajian and Maskit in \cite[Theorem 1.1]{basmajian}. Precise involution length for ${\rm PU}(n,1)$ is not known for arbitrary $n$. The involution length of  ${\rm PU}(2,1)$ is $4$ (see \cite[Theorem 1]{pw}) and when $n\geq 3$ the involution length of ${\rm PU}(n,1)$  is at most $8$ (see \cite[Theorem 2]{pw}). 

Even though the decompositions of isometries into involutions in the above groups are known to some extent, not much has been known for their quaternionic counterpart $\Sp(n,1)$. Here we specifically mean that the involution lengths in the above groups are known but the involution length is not known for $\PSp(n,1)$.   
Let $\mathbf{H}^n_{\mathbb{H}}$ be the $n$-dimensional quaternionic hyperbolic space. Let $\mathrm{Isom}(\mathbf{H}^n_{\mathbb{H}})$ denote the isometry group which consists of all the isometries of $\mathbf{H}^n_{\mathbb{H}}$, which is isomorphic to $\PSp(n,1)$ for $n>1$ (for example, see \cite{BH}). For $n=1$, we consider the index two subgroup $\PSp(1,1)$ of $\mathrm{Isom}(\mathbf{H}^1_{\mathbb{H}})$, which is the connected component of $\mathrm{Isom}(\mathbf{H}^1_{\mathbb{H}})$. In this paper, we give a complete classification of reversible and strongly reversible elements in $\Sp(n,1)$.
\begin{theorem}\label{mainth}
Every element of $\mathrm{Sp}(n, 1)$ is reversible.
\end{theorem} 
 However, in contrast to the real and complex hyperbolic isometries, reversibility does not imply strong reversibility in $\Sp(n,1)$. In fact, we shall see most elements in $\Sp(n,1)$ are not strongly reversible. Let $\Sp(n)$ denote the maximal compact subgroup of $\Sp(n,1)$. A complete classification of strongly reversible elements is not known even for this compact group. O'Farrell and Short raised this as an open problem in their book, see \cite[p. 91]{OS}.  
We also give a complete answer to this problem. We prove the following 
\begin{theorem} \label{strealspn}
An element $g\in \mathrm{Sp}(n)$ is strongly reversible if and only if every eigenvalue class of $g$ is either $\pm 1$ or of even multiplicity. 
\end{theorem}
The involution length of $\Sp(n)$ is known to be at most $6$, for example, see \cite[Theorem 5.9]{OS}. As a corollary to the above theorem, we improve the upper bound of the involution length of $\Sp(n)$. The only involutions in $\Sp(1)$ are $\pm 1$. For  $n \geq 2$, we have the following. 
\begin{corollary}\label{cor1}
For $n>1$, every element in $\Sp(n)$ can be expressed as a product of four (resp. five) involutions if $n$ is even (resp. $n$ is odd). 
\end{corollary} 
We apply the above theorem to give the following characterization of strongly reversible elements in $\Sp(n,1)$.  
Using conjugation classification~\cite{CG}, we know that semisimple elements in $\Sp(n,1)$ are classified into two types as {\it elliptic} and  {\it hyperbolic}. An element $g\in \Sp(n,1)$ which is not semisimple is called {\it parabolic}.  
However, it has a Jordan decomposition $g=g_{s}g_{u}$, where $g_{s}$ is elliptic, hence semisimple, 
and $g_{u}$ is unipotent. In particular, if a parabolic isometry is unipotent, then it has all eigenvalues $1$. A unipotent parabolic with minimal polynomial $(x-1)^2$ (resp. $(x-1)^3$) is called 
a \emph{vertical translation} (resp. a \emph{non-vertical translation}). For details, see the next section. For terminologies of the next theorem, we refer to Definition \ref{positive}.
\begin{theorem}\label{sr}
Suppose $g$ is an element of ${\rm Sp}(n,1)$.
\begin{enumerate}
\item{Let $g$ be  hyperbolic. Then $g$ is strongly reversible if and only if every positive eigenvalue class of $g$ is either $\pm 1$ or of even multiplicity, and the null eigenvalues of $g$ are real numbers.  }
\item {Let $g$ be  elliptic. Then $g$ is strongly reversible if and only if the 
eigenvalue of negative or indefinite type of $g$ is $1$ or $-1$ and every positive eigenvalue class of $g$ is either $\pm 1$ or of even multiplicity.  }
\item {Let $g$ be a vertical translation.  
Then $g$ is not strongly reversible. }
\item {Let $g$ be a non-vertical translation.  
Then $g$ is strongly reversible. }
\item {Let $g=g_sg_u$ be non-unipotent parabolic. Then $g$ is strongly reversible 
if and only if the null eigenvalue of $g$ is $1$ or $-1$,  the minimal 
polynomial of $g_u$ is $(x-1)^3$, and every positive eigenvalue class of $g$ is either $\pm 1$ or of even multiplicity.  }
\end{enumerate} 
\end{theorem}
However, when we project to the full isometry group, we have strong reversibility for every element. 
\begin{theorem}\label{strealpspn1}
Every element of $\mathrm{PSp}(n,1)$ is  a product of two involutions. 
\end{theorem} 

\subsubsection*{Structure of the paper}
In \secref{preliminaries}, we cover the preliminaries. In \secref{realstrealpspn}, we explore the reversible and strongly reversible elements in $\Sp(n)$ and $\mathrm{PSp}(n)$.
In \secref{rpf}, we give a complete description of the reversible and strongly reversible elements in $\Sp(n,1)$ and we prove one of our main \thmref{mainth} of this paper. Last two sections are devoted to the proof of \thmref{sr} and \thmref{strealpspn1} respectively.
\section{Preliminaries}\label{preliminaries}
In this section, we fix some notations and terminologies which will be used throughout this paper. Let $\mathbb{H}=\mathbb{R}\oplus \mathbb{R}i\oplus \mathbb{R}j\oplus \mathbb{R}ij$ be the real  quaternions. We identify the subspace $\R \oplus \R i$ of $\mathbb H$ with the standard complex plane in $\mathbb H$.

\subsection{Matrices over quaternions}  Let $\V$ be an $n$-dimensional  right vector space over $\mathbb H$.  Let $T$ be a right linear transformation of $V$. Then $T$ is represented by an $n \times n$ matrix over $\mathbb H$. Invertible linear maps of $\V$ are represented by invertible $n \times n$ quaternionic matrices. The group of all such linear maps is denoted by ${\rm GL}(n, \mathbb H)$. For more details on linear algebra over quaternions, see \cite{rodman}. In the following, we briefly recall the notions that will be used later on.

 \medskip Let $T \in \mathrm{GL}(n, \mathbb H)$ and $v \in \V$, $\lambda \in {\mathbb H}^{\times}$, are such that $T(v)=v \lambda$, then for $\mu \in {\mathbb H}^{\times}$ we have 
$$T(v \mu)=(v \mu) \mu^{-1} \lambda \mu.$$ Therefore eigenvalues of $T$ occur in similarity classes and if $v$ is a $\lambda$-eigenvector, then $v \mu \in v \mathbb H$ is a $\mu^{-1} \lambda \mu$-eigenvector.  Thus the eigenvalues are no more conjugacy invariants for $T$, but the similarity classes of eigenvalues are conjugacy invariant.   

Note that each similarity class of eigenvalues contains a unique pair of complex conjugate numbers.  Often we shall refer them as `eigenvalues', though it should be understood that our reference is towards their similarity classes. In places where we need to distinguish between the similarity class and a representative, we shall denote the similarity class of an eigenvalue representative $\lambda$ by $[\lambda]$. 
\subsection{The group $\Sp(n,1)$} 
In this section, we are following Chen-Greenberg~\cite{CG}.
Let $\V:=\mathbb{H}^{n,1}=\mathbb{H}^{n+1}$ be an $(n+1)$-dimensional right vector space over $\mathbb{H}$ equipped with a  $\mathbb{H}$-Hermitian form 
\[\Phi(z,w)=-\bar{z}_0w_0+\bar{z}_1w_1+\cdots+\bar{z}_nw_n,\] 
where $z=(z_0, z_1, \ldots, z_n), \; w=(w_0, w_1, \ldots, w_n)\in \mathbb{H}^{n+1}$. Therefore the matrix representation of $\Phi$ with respect to the standard basis $\{e_0, e_1, \ldots, e_n\}$ of $\mathbb{H}^{n+1}$ is $J=\diag (-1, 1, \ldots, 1)$. The {\it symplectic group} of signature $(n,1)$ is \[\mathrm{Sp}(n,1)=\{g\in \mathrm{GL}(n+1, \mathbb{H})\mid {}^{t}\bar{g}Jg=J\}.\]

If we restrict the $\mathbb{H}$-Hermitian form $\Phi$ on the orthogonal complement of the one-dimensional subspace $e_0 \mathbb{H}$, then the linear transformations preserving the restricted form is the following group 
\[\mathrm{Sp}(n):=\{g \in \mathrm{GL}(n,\mathbb{H}) \mid {}^{t}\bar{g}g=I_n\},\]  which is 
a compact subgroup of $\mathrm{GL}(n,\mathbb{H})$.  Restricting the positive-definite quaternionic Hermitian form over the standard complex subspace $\C^n \subset \mathbb {H}^n$, we get a copy of the compact unitary group as a subgroup of $\Sp(n)$, and we denote it by ${\rm U}(n)$. Thus
\[\mathrm{U}(n):=\{g \in \mathrm{GL}(n,\mathbb{C}) \subset \mathrm{GL}(n,\mathbb{H})  \mid {}^{t}\bar{g}g=I_n\}.\]

\begin{remark}
Let $P\in \mathrm{M}(n,\mathbb{H})$, then $P=A+Bj$, where $A,B \in \mathrm{M}(n,\mathbb{C})$.
Let \[\varphi : \mathrm{M}(n,\mathbb{H})\hookrightarrow \mathrm{M}(2n, \mathbb{C})\] defined by 
$\varphi(P)=\varphi(A+Bj)=\begin{pmatrix}A&B\\-\bar{B}&\bar{A}\end{pmatrix}$. Observe that the above map $\varphi$ is an injective group homomorphism from $\mathrm{GL}(n, \mathbb{H})$ to $\mathrm{GL}(2n, \mathbb{C})$. Observe that, 
\begin{itemize}
		\item  $\varphi(I_nj)=\begin{pmatrix}0&I_n\\-I_n&0\end{pmatrix}=:\beta$.
		\item $\varphi(P^*)=\varphi(P)^*$.
		\item $P$ is symplectic if and only if $\phi(P)$ is unitary.
	\end{itemize}
\end{remark}
 So, we have 
\[\mathrm{Sp}(n)\cong \mathrm{U}(2n) \cap \mathrm{Sp}(2n, \mathbb{C}),\] 
where $\mathrm{Sp}(2n, \mathbb{C}):=\{g \in \mathrm{GL}(2n, \mathbb{C}) \mid {}\tr g\beta g=\beta \}.$
Here $\mathrm{U}(2n)$ and $\mathrm{Sp}(n)$ are compact groups and defined over $\mathbb{R}$.
The following definition will be used later.  
\begin{definition}\label{multiplicity}
	Let $g\in \mathrm{Sp}(n)$. Let distinct eigenvalues of $g$ be represented by  $e^{i\theta_1}, e^{i\theta_2}, \ldots, e^{i\theta_m}$, $m \leq n$. The right vector space $\mathbb{H}^n$ has the following orthogonal decomposition into eigenspaces:
	\[ \mathbb{H}^n=\V_{\theta_1}\oplus \V_{\theta_2}\oplus\cdots\oplus \V_{\theta_m},\]
	where $\V_{\theta_l}=\{v\in \mathbb{H}^n\mid gv=ve^{i\theta_l}\}$ for $1\leq l\leq m$.
	We define {\it multiplicity } of $e^{i\theta_l}:=\mathrm{dim}\;(\V_{\theta_l})$. Equivalently, it is the repetition of the eigenvalue $e^{i\theta_l}$ in a diagonal form, up to conjugacy,  of $g$. 
\end{definition}
\subsection{Hyperbolic space \texorpdfstring{$\mathbf{H}^n_{\mathbb{H}}$}{}}
 Define  
\begin{align*}
\V_0:=\{v\in \V \mid \Phi(v,v)=0\}, \\
\V_{+}:=\{v\in \V \mid \Phi(v,v) > 0\}, \\
\V_{-}:=\{v\in \V \mid \Phi(v,v) < 0\}.
\end{align*}
Let $\mathbb{P}(\V)$ be the quaternionic projective space, i.e., 
$\mathbb{P}(\V)={\V\smallsetminus \{0\}}/{\sim}$, where 
$u \sim v$ if there exists $\lambda \in \mathbb{H}^{\times}$ such that 
$u=v\lambda$. Here $\mathbb{P}(\V)$ is equipped with the quotient topology and the 
quotient map is $\Pi : \V\smallsetminus \{0\} \rightarrow \mathbb{P}(\V)$. 
The $n$-dimensional quaternionic hyperbolic space is defined to be 
$\mathbf{H}_{\mathbb H}^{n}:=\Pi (\V_{-})$. 
The boundary $\partial \mathbf{H}_{\mathbb H}^{n}$ in 
$\mathbb{P}(\V)$ is $\Pi (\V_0)$.
The isometry group $\mathrm{Sp}(n,1)$ acts by isometries on $\mathbf{H}_{\mathbb H}^{n}$. The actual group of 
isometries of $\mathbf{H}_{\mathbb H}^{n}$ is $\mathrm{PSp}(n,1)={\mathrm{Sp}(n,1)}/{\mathcal{Z}(\mathrm{Sp}(n,1))}$, where $\mathcal{Z}(\mathrm{Sp}(n,1))=\{\pm I_{n+1}\}$ is the center. 
Thus an isometry $g$ of $\mathbf{H}_{\mathbb H}^{n}$ 
lifts to a symplectic transformation $\tilde{g} \in \mathrm{Sp}(n,1)$. The fixed points of $g$ 
correspond to eigenvectors of $\tilde{g}$.
\begin{definition} A subspace $\W$ of $\V$ is called  \emph{space-like}, \emph{light-like}, or \emph{ indefinite} if the Hermitian form restricted to $\W$ is positive-definite, degenerate, or non-degenerate but indefinite respectively. If $\W$ is an indefinite subspace of $\V$, then the orthogonal complement $\W^{\perp}$  is space-like. 
\end{definition} 
\begin{definition}\label{positive}
An eigenvalue $\lambda$ of an element $g\in \Sp(n,1)$ is called \emph{positive type} (resp. \emph{negative type}) if every eigenvector in $\V_{\lambda}$ is in $\V_{+}$ (resp. $\V_{-}$). The eigenvalue $\lambda$ is called \emph{null} if the eigenspace $\V_{\lambda}$ is degenerate. The eigenvalue $\lambda$ is called \emph{indefinite type} if the eigenspace $\V_{\lambda}$ contains vectors in $\V_{+}$ and $\V_{-}$.
\end{definition}
Accordingly, a similarity class of eigenvalues $[\lambda ]$ is \emph{null}, 
\emph{positive} or \emph{negative} according to its representative $\lambda $ is null, 
positive or negative respectively.

\subsection{Cayley transform}
If we change the standard basis by $\{\widehat{e}_0, \widehat{e}_1, \ldots, \widehat{e}_n\}$, where $\widehat{e}_0:=\frac{e_0-e_1}{\sqrt{2}}, \widehat{e}_1:=\frac{e_0+e_1}{\sqrt{2}}$ and $\widehat{e}_l:=e_l$ for $2\leq l\leq n$, then we get a change of basis matrix, which is the following
\[P:=\begin{pmatrix}1/\sqrt{2}&1/\sqrt{2}&0\\-1/\sqrt{2}&1/\sqrt{2}&0\\0&0&I_{n-1}\end{pmatrix}.\]
Then define $C:=P^{-1}$, which is called the \textit{Cayley transform}. 
Now \[\widehat{J}:={}^{t}PJP=\begin{pmatrix}0&-1&0\\-1&0&0\\0&0&I_{n-1}\end{pmatrix}.\] Therefore the corresponding $\mathbb{H}$-Hermitian form is \[\widehat{\Phi}(z,w)=-(\bar{z}_0w_1+\bar{z}_1w_0)+\bar{z}_2w_2+\cdots+\bar{z}_nw_n.\] The corresponding symplectic group is 
\[\widehat{\mathrm{Sp}}(n,1)=\{g\in \mathrm{GL}(n+1, \mathbb{H})\mid {}^{t}\bar{g}\widehat{J}g=\widehat{J}\}.\] Therefore we have  
$\widehat{\mathrm{Sp}}(n,1)=P^{-1}\mathrm{Sp}(n,1)P$, since  $\widehat{J}={}^{t}PJP$.
The Hermitian form $\widehat \Phi$ gives the \emph{Siegel domain} model of the quaternionic hyperbolic space. 

\subsection{Classification of isometries}\label{ehp} The non-identity  elements of $\mathrm{Sp}(n, 1)$ can be classified into three disjoint classes  depending on their fixed points.  An isometry $g$ is called \emph{elliptic} if it has a fixed point in $\mathbf{H}_{\mathbb H}^{n}$. 
It is called \emph{parabolic} if it has exactly one fixed point 
on the boundary $\partial \mathbf{H}_{\mathbb H}^{n}$, and is called \emph{hyperbolic} (or \emph{loxodromic}) if it has exactly two fixed points on the boundary $\partial \mathbf{H}_{\mathbb H}^{n}$. Notice that, if two elements are conjugate, then they have the same (elliptic, parabolic, or hyperbolic) type.

\begin{lemma}\label{hycl}{\rm (Chen-Greenberg )\cite[Theorem 3.4.1]{CG}}
\begin{enumerate} 
\item 
 Two elliptic elements are conjugate if and only if they have the same negative class of eigenvalues, and the same positive classes of eigenvalues.

\item Two loxodromic elements are conjugate if and only if they have the same similarity classes of eigenvalues.

\item Two parabolic elements are conjugate if and only if their semisimple (elliptic) and unipotent components are conjugate.
\end{enumerate} 
\end{lemma}
The following follows from the conjugacy classification in $\widehat \Sp(n,1)$, e.g. \cite{CG}.
\begin{lemma}
Let $g$ be a parabolic element in $\widehat \Sp(n,1)$. Then, up to conjugacy, $g$ is one of the following forms:
\begin{enumerate}
\item   
$g =\begin{pmatrix}\lambda u_V & 0 \\0& B\end{pmatrix},$ where $B\in \Sp(n-1)$, and   $u_V=\begin{pmatrix}1&0\\s&1
\end{pmatrix}\in \widehat{\Sp}(1,1)$  with $s+\bar{s}=0$, and $|\lambda|=1$.
\item $g=\begin{pmatrix}\lambda u_{NV}&0\\0&B\end{pmatrix}$, where 
$B\in\Sp(n-2)$, and $u_{NV}=\begin{pmatrix}1&0&0\\s&1&\bar{a}
\\a&0&1\end{pmatrix}\in \widehat{{\rm Sp}}(2,1)$ where $a, s \in \mathbb{H}^{\times}$  with $s+\bar{s}=|a|^2$, and $|\lambda|=1$.
\end{enumerate}
\end{lemma}

\section{Reversible and Strong Reversible Elements in \texorpdfstring{$\mathrm{PSp}(n)$}{}}\label{realstrealpspn}
\subsection{Reversible elements in $\Sp(n)$}
Part (1) of the following result is known, see \cite[Theorem 5.7]{OS}. We will give a shorter and different proof that will imply part (2) of the theorem that seems unavailable in the literature. 
\begin{proposition}\label{realspn}
\noindent
\begin{enumerate} 
\item Every element of $\Sp(n)$ is reversible. 
\item Every element of $\mathrm{PSp}(n)$ is strongly reversible.
\end{enumerate} 
\end{proposition}
\begin{proof} We know that every element of $\mathrm{Sp}(n)$ is semisimple. So up to conjugacy we can write \[g=\mathrm{diag}(\lambda_1, \lambda_2, \ldots, \lambda_n),\] where $\lambda_k\in \mathbb{C}$ with $|\lambda_k|=1$ for all $k=1,2, \ldots, n$. Now observe that  $j\lambda_k=\bar{\lambda}_k j$ for all $k=1, 2, \ldots, n$. Therefore \[g^{-1}=hgh^{-1},\] where $h=\mathrm{diag} (j, j, \ldots, j)\in \mathrm{Sp}(n)$. Hence every element of $\mathrm{Sp}(n)$ is reversible. 

Now observe that $h^2=-I_n$, so every element of $\mathrm{PSp}(n)$ is strongly reversible. \end{proof}
In \cite[Chapter 5]{OS}, O'Farrell and Short pose the problem of characterizing strongly  reversible elements in $\mathrm{Sp}(n)$.  The following result answers this problem. But  before proving it, we recall the concept of projective points from \cite[Section 3]{gk1}. 

\subsection{Projective points} 
\medskip Let $g \in \Sp(n)$. Let $\lambda \in \H\setminus\R$ be a chosen eigenvalue of $g$ in the similarity class $[\lambda]$.  We may assume that $[\lambda]$ has  multiplicity one, i.e., the $[\lambda]$-eigenspace has dimension one. Thus, we can identify the eigenspace $\V_{[\lambda]}$ with $\H$. Consider the $\lambda$-\emph{eigenset}: $S_{\lambda}=\{ x \in \V \ | \ gx =x \lambda \}$. Note that this set is the complex line $x \mathcal{Z}_{\mathbb{H}}(\lambda)$ in $\H$. Note that for $q \in \H \setminus \R$, $S_{q \lambda q^{-1}}=S_{\lambda}q$. 

\medskip  Identify $\H$ with $\C^2$. Two non-zero quaternions $q_1$ and $q_2$ are equivalent if  $q_2=q_1 c$, $c \in \C \setminus 0$. This equivalence relation projects $\H$ to the one-dimensional complex projective space $\C \mathbb{P}^1$.  Thus, $\V_{[\lambda]}=\sqcup_{q \in \C \mathbb{P}^1} S_{q \lambda q^{-1}}$. The $\C  \mathbb{P}^1$ associated to $[\lambda]$ in this way will be called the  $[\lambda]$-projective sphere. 

\subsection{Proof of \thmref{strealspn}}
Suppose $g=\diag\;(\lambda_1, \lambda_2, \ldots, \lambda_n)$ is strongly reversible. Then $g^{-1}=hgh^{-1}$ for some $h\in \mathrm{Sp}(n)$ such that $h^2=I_n$. Then we have the following orthogonal decomposition: 
\[\mathbb{H}^n=\V_{[\lambda_1]}\oplus \V_{[\lambda_2]}\oplus\cdots \oplus \V_{[\lambda_n]},\]
where $\mathrm{dim}\V_{[\lambda_l]}=1$ for all $l=1, 2, \ldots, n$. Note that 
$\V_{[\lambda_l]}=\sqcup_{q \in \C \mathbb{P}^1} S_{q \lambda_l q^{-1}}$. Now, if $v$ is an eigenvector of $g$ in the $\lambda_l$-eigenset $S_{\lambda_l}$, then
$$g^{-1}(h(v))=hg(v)=h(v) \lambda_l,$$
i.e., $g(h(v))=h(v) \lambda_l^{-1}$. Therefore, $h(v)$ is an eigenvector of $g$ in the $\lambda_l^{-1}$-eigenset $S_{\lambda_l^{-1}}$. 
		Thus, either $h$ maps $\V_{[\lambda_l]}$ to $\V_{[\lambda_t]}$ with $\lambda_t=\overline{\lambda}_l$ for some $t$, or, $h$ preserves $\V_{[\lambda_l]}$ and permutes the $S_{\lambda_l}$'s inside $\V_{[\lambda_l]}$. In the first case $[\lambda]$ is of even multiplicity. In the second case, $h$ is an involution on $\V_{[\lambda_l]}$, and acts as an orientation-preserving involution on the $[\lambda_l]$-projective sphere.  Hence, $h|_{\V_{[\lambda_l]}}$ is $\pm 1$, that implies  $\lambda_l=\pm 1$. 

Conversely, suppose the assertion holds. Then, up to conjugacy, $g$ is conjugate to an element $\tilde g$ of ${\rm U}(n)$ that has the property that if $\lambda \neq \pm 1$ is an eigenvalue, then so is $\lambda^{-1}$. Consequently, it is strongly reversible in ${\rm U}(n)$. Hence $g$ is strongly reversible in $\Sp(n)$. 

This completes the proof.

\subsection{Proof of \corref{cor1}}
Let $g \in \Sp(n)$. Up to conjugacy, we can choose $g$ to be the diagonal element
$$g=\begin{pmatrix} \lambda_1 & & &  \\ & \lambda_2 & &  \\ & & \ddots & \\ & & & \lambda_n \end{pmatrix},$$
where $\lambda_l$'s are complex numbers with $|\lambda_l|=1$. 
First, we prove this result when $n$ is even, say $n=2m \; (m\geq 1)$.
Note that we can write $g=\alpha_1 \alpha_2$, where
\begin{equation}\label{r3} \alpha_1=\hbox{diag}(\lambda_1, \overline \lambda_1, ~ \lambda_1\lambda_2 \lambda_3, \overline \lambda_1 \overline \lambda_2 \overline \lambda_3, ~\ldots, ~ \lambda_1 \lambda_2 \cdots \lambda_{2k+1}, \overline \lambda_1 \overline \lambda_2\cdots \overline \lambda_{2k+1}, \ldots) \end{equation}
\begin{equation} \label{r4} \alpha_2=\hbox{diag}(1, ~\lambda_1 \lambda_2, \overline \lambda_1 \overline \lambda_2, ~\ldots, ~\lambda_1 \lambda_2 \cdots \lambda_{2k}, \overline\lambda_1 \overline \lambda_2 \cdots \overline \lambda_{2k}, \ldots).\end{equation}
Since $\lambda_j$ and $\bar \lambda_j$ belong to the same similarity class, 
by the previous theorem, we see that both $\alpha_1$ and $\alpha_2$ are strongly reversible in $\Sp(n)$. Thus, $g$ can be written as a product of four involutions. 

When $n$ is odd, say $n=2m+1=2(m-1)+3 \; (m\geq 1)$. 
Since $g$ is semisimple, then the right vector space $\mathbb{H}^n$ has an orthogonal decomposition into $g$-invariant subspaces: $\mathbb{H}^n=\U\oplus \W$ with $\dim \U=3$ and $\dim \W=n-3$. Then $g|_{\U}$ can be considered as an element in $\Sp(3)$ and $g|_{\W}$ as an element in $\Sp(n-3)$ and $g=g|_{\U}\oplus g|_{\W}$. 
Now from a result of Djokovi\'{c} and Malzan, see \cite[Theorem 3]{dm}, the involution length of $\Sp(3)$ is at most five, i.e., $g|_{\U}$ can be written as a product of five involutions, say $g|_{\U}=r_1r_2r_3r_4r_5$, where $r_l^2=1$ for $l=1,2,3,4,5$. In view of the previous case, $g|_{\W}\in \Sp(2(m-1))$ can be expressed as a product of four involutions, say $g|_{\W}=s_1s_2s_3s_4$, where $s_k^2=1$ for $k=1,2,3,4$. 
Hence $g=(r_1\oplus s_1)(r_2\oplus s_2)(r_3\oplus s_3)(r_4\oplus s_4)(r_5\oplus I_{2(m-1)})$, where $r_k\oplus s_k$ are involutions. 
Therefore the involution length of $\Sp(2m+1)$ is at most $5$.

This completes the proof. 

\section{Reversible Elements in \texorpdfstring{$\mathrm{Sp}(n,1)$}{}} \label{rpf} 
For parabolic and hyperbolic elements, we will work in $\widehat{\mathrm{Sp}}(n,1)$. To prove the main theorem we need the following lemmas.
\begin{lemma}[Vertical translation]\label{lemma1}
Let $u_V:=\begin{pmatrix}1&0\\s&1\end{pmatrix}\in \widehat{\mathrm{Sp}}(1,1)$, where $s\in \mathrm{Im}(\mathbb{H})$. Then $u_V$ is reversible but not strongly reversible.	
\end{lemma}
\begin{proof}
Let $s=s_1i+s_2j+s_3ij\in \mathrm{Im}(\mathbb{H})$, then 
\[s=q(ri)q^{-1},\] where $r=\sqrt{s_1^2+s_2^2+s_3^2}$ and $q=(r+s_1)-s_3j+s_2ij$ (see~\cite[Lemma 1.2.2]{CG}). Therefore we have 
\[\begin{pmatrix}1&0\\s&1\end{pmatrix}=\begin{pmatrix}x&0\\0&x\end{pmatrix}\begin{pmatrix}1&0\\ ri&1\end{pmatrix}\begin{pmatrix}x&0\\0&x\end{pmatrix}^{-1},\] 
where $x=\frac{q}{|q|}$ (observe that $\begin{pmatrix}x&0\\0&x
\end{pmatrix}\in \widehat{\mathrm{Sp}}(1,1)$ as $|x|=1$). Now we have 
\[\begin{pmatrix}1&0\\ri&1\end{pmatrix}^{-1}=\begin{pmatrix}j&0\\0&j\end{pmatrix}\begin{pmatrix}1&0
\\ri&1\end{pmatrix}\begin{pmatrix}j&0\\0&j\end{pmatrix}^{-1}.\]
Therefore we have 
\[u_V^{-1}=h_Vu_Vh_V^{-1},\] where $h_V=\begin{pmatrix}xjx^{-1}&0\\0&xjx^{-1}\end{pmatrix}$. Hence $u_V$ is reversible. 

Now if there exists an $h=\begin{pmatrix}a&b\\c&d\end{pmatrix}\in \widehat{\mathrm{Sp}}(1,1)$ such that \[\begin{pmatrix}1&0\\\lambda i&1\end{pmatrix}^{-1}=\begin{pmatrix}a&b\\c&d\end{pmatrix}\begin{pmatrix}1&0\\\lambda i&1\end{pmatrix}\begin{pmatrix}a&b\\c&d\end{pmatrix}^{-1},\] then $h$ has the following form \[h=\begin{pmatrix}a&0\\c&\bar{a}^{-1}\end{pmatrix},\] where $\bar{a}c+\bar{c}a=0$ and $iai=\bar{a}^{-1}$. If possible suppose that $h$ is an involution, i.e., $h^2=I_2$, then $a=\pm1$. This is a contradiction, since $a\notin \mathbb{R}$. Hence $u_V$ is not strongly reversible.
\end{proof}
\begin{remark} 
Observe that in the above proof $h_V^2=-I_2$. 
\end{remark} 
\begin{lemma}\label{parabolicsp11}
	Let $g=\begin{pmatrix}\lambda &0\\\lambda s&\lambda\end{pmatrix}\in \widehat{\mathrm{Sp}}(1,1)$ be an arbitrary  parabolic element, where $\lambda, s\in \mathbb{C}$ with $|\lambda|=1$ and $\mathrm{Re}(s)=0$. Then $g$ is reversible.
\end{lemma}
\begin{proof}
	Without loss of generality, we can assume $s=ri$ with $r\in \mathbb{R}^{\times}$ and  $\lambda=e^{i\theta}$. Now pick $h_1=\begin{pmatrix}
	j&0\\0&j\end{pmatrix}\in \widehat{\mathrm{Sp}}(1,1)$, then by direct computation, one can check that $h_1gh_1^{-1}=g^{-1}$.
\end{proof}
\begin{remark} 
	Observe that in the above proof $h_1^2=-I_2$. 
\end{remark} 
\begin{lemma}[Non-vertical translation]\label{lemma2}
	Let $u_{NV}=\begin{pmatrix}1&0&0\\s&1&\bar{a}\\a&0&1	\end{pmatrix}\in \widehat{\mathrm{Sp}}(2,1)$, where $s,a\in \mathbb{H}^{\times}$ with $s+\bar{s}=|a|^2$. Then the element $u_{NV}$ is strongly reversible. In particular, $u_{NV}$ is reversible. 
\end{lemma}
\begin{proof} 
 Choose an involution $h_{NV}$ in $\widehat{\mathrm{Sp}}(2,1)$ such that
$$h_{NV}=\begin{pmatrix} 1 & 0 & 0 \\ \frac{1}{2} |d|^2 & 1 & d \\ -\bar d & 0 & -1 \end{pmatrix},$$
where $d$ is chosen so that $s+da$ is a real number. By direct computation, one can check that $h_{NV} u_{NV} h_{NV}^{-1}=u_{NV}^{-1}$. 
This completes the proof. 
\end{proof} 
\begin{lemma}\label{parabolicsp21}
	Let $g=\begin{pmatrix}\lambda&0&0\\\lambda s&\lambda&\lambda \bar{a}\\\lambda a&0&\lambda\end{pmatrix}\in \widehat{\mathrm{Sp}}(2,1)$ be an arbitrary  parabolic element, where $s, a\in \mathbb{R}(\lambda), |\lambda|=1$ and $s+\bar{s}=|a|^2$. Then $g$ is reversible. 
\end{lemma}
\begin{proof}
	Without loss of generality we may assume that $a, s, \lambda \in \mathbb{C}$ with $\lambda=e^{i\theta}$. Now pick $h_2=\begin{pmatrix}j&0&0\\ij&j&0\\0&0&-ij
	\end{pmatrix}\in \widehat{\Sp}(2,1)$. By direct computation, one can see that $h_2gh_2^{-1}=g^{-1}$. 
\end{proof}
\begin{remark} 
	Observe that in the above proof $h_2^2=-I_3$. 
\end{remark} 
\begin{lemma} \label{sp1}
Let $g$ be a hyperbolic element  in $\widehat{\Sp}(1,1)$. Then $g$ is reversible. Further, $g$ is  strongly reversible if and only if both eigenvalues of $g$ are real numbers. 
\end{lemma}
\begin{proof}
	Let $g=\begin{pmatrix}re^{i\theta}&0\\0&r^{-1}e^{i\theta}\end{pmatrix}$, where $r>1$ and $0\leq \theta \leq \pi$. Then $g^{-1}=hgh^{-1}$, where $h=\begin{pmatrix}0&j\\j&0\end{pmatrix}\in \widehat{\mathrm{Sp}}(1,1)$, which proves that $g$ is reversible.   
	
	For the second part, let $g$ be a strongly reversible element. Then $g^{-1}=hgh^{-1}$ for some $h\in \widehat{\mathrm{Sp}}(1,1)$ with $h^2=I_2$. Suppose the eigenvalues of $g$ are not real numbers. Then we can assume $\theta \neq 0, \pi$. Now let $h=\begin{pmatrix}a&b\\c&d\end{pmatrix}$. Then we get 
	\begin{align}
	a&=r^2e^{i\theta}ae^{i\theta}\label{hypeq1}\\
	b&=e^{i\theta}be^{i\theta}\label{hypeq2}\\
	c&=e^{i\theta}ce^{i\theta}\label{hypeq3}\\
	d&=r^{-2}e^{i\theta}de^{i\theta}\label{hypeq4}.
	\end{align}
 From Equation \eqnref{hypeq1} we get  $a=a_0+a_1i+a_2j+a_3ij=r^2e^{i\theta}(a_0+a_1i+a_2j+a_3ij)e^{i\theta}$, which implies that $a_2=0=a_3$ as $r^2\neq 1$. If $a_0\neq 0$ (resp. $a_1\neq 0$), then $1=r^2e^{2i\theta}$, which implies that $\theta=0$ or $\pi/2$ or $\pi$. But $\theta \neq 0, \pi$ so $\theta=\pi/2$, which implies that $r^2=-1$ a contradiction. Therefore $a=0$.  
Similarly, from Equation \eqnref{hypeq4}, we get $d=0$.

 From Equation \eqnref{hypeq3}, we have  $c=c_0+c_1i+c_2j+c_3ij=e^{i\theta}(c_0+c_1i+c_2j+c_3ij)e^{i\theta}$ which implies that $c_0=0=c_1$ as $\theta \neq 0, \pi$.
 
 Again, we have $bc=1=cb=\bar{c}b$ since $h^2=I_2$ and $h\in \widehat{\mathrm{Sp}}(1,1)$. Therefore $c=\bar{c}$, i.e., $c\in \mathbb{R}\setminus \{0\}$. This is a contradiction to the fact that $c=c_2j+c_3ij$. Therefore both eigenvalues of $g$ are real numbers.
	
Conversely, if all the eigenvalues of $g$ are real numbers, then $\theta=0$ or $\pi$, i.e., $g=\begin{pmatrix}r&0\\0&r^{-1}\end{pmatrix}$ or $\begin{pmatrix}-r&0\\0&-r^{-1}\end{pmatrix}$. By direct computation, we get $g^{-1}=hgh^{-1}$, where $h=\begin{pmatrix}0&1\\1&0\end{pmatrix}\in \widehat{\mathrm{Sp}}(1,1)$ such that $h^2=I_2$. Therefore, $g$ is strongly reversible.
\end{proof}

\subsection{Proof of \thmref{mainth}}  \begin{enumerate}  
\medskip \item[(i)] {\it Let $g$ be elliptic in $\Sp(n,1)$}.  From the conjugation classification, we know that $g$ is semisimple with eigenvalues of norm $1$. It has $n$ similarity classes of positive eigenvalues (which may not all be distinct) and one similarity class of negative eigenvalue (which may coincide with one of the positive classes), for details, see \cite[Lemma 3.2.1 and Proposition 3.2.1]{CG}. So, up to conjugacy, we can assume \[g=\mathrm{diag}(\lambda_0, \lambda_1, \ldots, \lambda_n),\] where $\lambda_k\in \mathbb{C}$ with $|\lambda_k|=1$ for $k=0, 1, \ldots, n$. Since $j\lambda_k=\overline{\lambda}_kj$ for $\lambda_k\in \mathbb{C}$ then we have \[g^{-1}=hgh^{-1},\] where $h=\mathrm{diag} (j, j,\ldots, j)\in \mathrm{Sp}(n, 1)$. Therefore every elliptic element of $\mathrm{Sp}(n, 1)$ is  reversible. Note that $h^2=-I_{n+1}$ in $\Sp(n,1)$.

\medskip In the remaining part of the proof, we shall use the Siegel domain model and assume $g \in \widehat \Sp(n,1)$. 

\medskip \item[(ii)] {\it Let $g$ be hyperbolic}. Let $\lambda$ be the (null) eigenvalue class of $g$
with $|\lambda|>1$. Then $\V$ has a decomposition into 
$g$-invariant  {orthogonal} subspaces: $\V=\U \oplus \W$, where 
$\U$ is the direct sum of the one-dimensional  eigenspaces
$\V_{\lambda}$ and $\V_{\overline{\lambda}^{\,-1}}$
and $\W$ is the space-like orthogonal complement to $\U$.  The Hermitian form 
restricted to $\U$ has the signature $(1,1)$, hence $g|_{\U}$ can be considered 
as a transformation in $\widehat{{\rm Sp}}(1,1)$ and $g|_{\W}$ as an element in $\Sp(n-1)$, and $g=g|_{\U} \oplus g|_{\W}$. 
Now it follows from \lemref{sp1} and  \propref{realspn} that \[g^{-1}=hgh^{-1},\] where $h=\begin{pmatrix}h_1&0\\0&h_2\end{pmatrix}$ with $h_1=\begin{pmatrix}0&j\\j&0\end{pmatrix}$ and $h_2=I_{n-1}j=\mathrm{diag} (j, j, \ldots, j)$. Note that $h^2=-I_{n+1}$ in $\widehat{\Sp}(n, 1)$. 

\medskip \item[(iii)] {\it Let $g$ be a translation}. Again from conjugation classification, there are exactly two conjugacy classes of unipotent parabolic elements. One is the vertical translation, denoted by $U_V=\begin{pmatrix}u_V&0\\0&I_{n-1}\end{pmatrix}$, with minimal polynomial $(x-1)^2$. The other one is the non-vertical translation, denoted by $U_{NV}=\begin{pmatrix}u_{NV}&0\\0&I_{n-2}\end{pmatrix}$, whose minimal polynomial is $(x-1)^3$. 
Therefore we have $U_V^{-1}=H_VU_VH_V^{-1}$, where $H_V=\begin{pmatrix}h_V&0\\0&I_{n-1}\end{pmatrix}$ and $h_V$ is defined in \lemref{lemma1}. Also, we have $U_{NV}^{-1}=H_{NV}U_{NV}H_{NV}^{-1}$, where $H_{NV}=\begin{pmatrix}h_{NV}&0\\0&I_{n-2}\end{pmatrix}$ and $h_{NV}$ is described in \lemref{lemma2}.
Therefore unipotent elements are reversible. 

\medskip \item[(iv)] {\it Let $g$ be parabolic}. From the conjugation classification, see \cite[Proposition 3.4.1]{CG},  we know that  $g\in \widehat{\mathrm{Sp}}(n,1)$ has the Jordan decomposition $g=g_sg_u=g_ug_s$, where $g_s$ is a unique elliptic element and $g_u$ is a unique unipotent parabolic element. If $g$ is parabolic, then ${\mathbb H}^{n,1}$ has a $g$-invariant orthogonal decomposition:  ${\mathbb H}^{n,1}=\U \oplus \W$, where $\dim \U=2$ or $3$, $g|_{\U}$ is indecomposable, i.e., $\U$ cannot be written as a sum of $g$-invariant subspaces, and $g|_{\W}$ acts on $\W$ as an element of $\Sp(n-1)$ or $\Sp(n-2)$. Further, if $\lambda$ represents the null eigenvalue of $g$, then $g$ has minimal polynomial $(x-\lambda)^l$, where $l=2$ or $3$. Then, up to conjugacy, \[g=\begin{pmatrix}g_1&0\\0&g_2\end{pmatrix},\] where $g_1\in \widehat{\Sp}(1,1)$ or $\widehat{\Sp}(2,1)$ and $g_2\in \Sp(n-1)$ or $\Sp(n-2)$. Now we have $g^{-1}=h'gh'^{-1}$, where $h'=\begin{pmatrix}h_1&0\\0&h\end{pmatrix}$ or $h'=\begin{pmatrix}h_2&0\\0&h\end{pmatrix}$ with $h_1$ as in \lemref{parabolicsp11}, $h_2$ as in \lemref{parabolicsp21} and $h$ as in \propref{realspn}. Note that $h'^2=-I_{n+1}$ in $\widehat{\Sp}(n,1)$.
\end{enumerate}
This completes the proof. 

\section{Proof of \thmref{sr}} \label{streal}

(1) Suppose $g$ is hyperbolic. Let $\lambda$ be the (null) eigenvalue class of $T$
with $|\lambda|>1$. Then $\V$ has a decomposition into 
$g$-invariant  {orthogonal} subspaces: $\V=\U \oplus \W$, where 
$\U$ is the direct sum of the one-dimensional  eigenspaces
$\V_{\lambda}$ and $\V_{\overline{\lambda}^{\,-1}}$
and $\W$ is the space-like orthogonal complement to $\U$.  The Hermitian form 
restricted to $\U$ has the signature $(1,1)$, hence $g|_{\U}$ can be considered 
as a transformation in $\widehat{\Sp}(1,1)$ and $g|_{\W}$ as an element in $\Sp(n-1)$. 
The result now follows from \lemref{sp1} and \thmref{strealspn}. 
 
\medskip \noindent (2) Suppose $g$ is elliptic. Then $g$ has a negative or indefinite type eigenvalue $\lambda$. Let $\mathrm{L}_{\lambda}$ be a one-dimensional subspace spanned by the corresponding eigenvector. In the orthogonal complement $\mathrm{L}_{\lambda}^{\perp}$, $g$ restricts to an element in $\Sp(n)$. Thus, up to conjugacy, we may consider $g$ as: 
$g=\begin{pmatrix} \lambda & 0 \\ 0 & g_1 \end{pmatrix}$, 
where $g_1 \in \Sp(n)$. 
It is easy to see that the only strongly reversible elements of $\Sp(1)$ are $1$ and $-1$.  The result now follows from \thmref{strealspn}. 

\medskip \noindent (3) Follows from \lemref{lemma1}, and (4) follows from \lemref{lemma2}. 

\medskip \noindent  (5)  Suppose $g$ is parabolic. Then  $g$ has the Jordan decomposition $g=g_s g_u$, where $g_s$ is semisimple, $g_u$ is unipotent, and $g_s g_u=g_u g_s$. If  $g$ is strongly reversible, then clearly $g_s$ and $g_u$ are strongly reversible.  
The null eigenvalue $\lambda$ of $g$ will be a negative eigenvalue for $g_s$, and hence the assertion necessarily follows from  (2) and the unipotent cases. 

Conversely, suppose the hypothesis holds. If $g$ is parabolic, then ${\mathbb H}^{n,1}$ has a $g$-invariant orthogonal decomposition:  ${\mathbb H}^{n,1}=\U \oplus \W$, where, $\dim \U=2$ or $3$, $g|_{\U}$ is indecomposable, i.e., $\U$ cannot be written as a sum of $g$-invariant subspaces, and $g|_{\W}$ acts on $\W$ as an element of $\Sp(n-1)$ or $\Sp(n-2)$. Further, if $\lambda$ represents the null eigenvalue of $g$, then $g$ has minimal polynomial $(x-\lambda)^l$, where $l=2$ or $3$. The given hypothesis implies that $g|_{\U}$ and $g|_{\W}$ are strongly reversible. Hence $g$ is strongly reversible. 

This completes the proof. 
\section{Proof of \thmref{strealpspn1}}\label{strealpsp}

From the proof of \thmref{mainth}, we see that for $g$ elliptic or hyperbolic, $g=hg^{-1} h^{-1}$, where $h^2=-I_{n+1}$ in $\Sp(n,1)$. Hence $h^2=I_{n+1}$ in $\mathrm{PSp}(n,1)$. Thus, $g$ is strongly reversible in $\mathrm{PSp}(n,1)$. 

For $g$ a vertical translation, choose $H_V'=\begin{pmatrix}h_V&0\\0&I_{n-1}j\end{pmatrix}$ in $\widehat{\Sp}(n,1)$, where $h_V$ is defined in \lemref{lemma1}. Hence $H_V'^2=-I_{n+1}$ projects to the identity element. We have already seen in \lemref{lemma2} that a non-vertical translation is strongly-reversible. Consequently, every translation is strongly reversible in $\mathrm{PSp}(n,1)$. 

An arbitrary parabolic element is strongly reversible in $\mathrm{PSp}(n,1)$ follows from the last part of the proof of \thmref{mainth}, and \lemref{parabolicsp11} and \lemref{parabolicsp21}. 

This completes the proof.

\medskip
\textbf{Acknowledgement:} The authors would like to thank Sagar B. Kalane of IISER Mohali for many helpful discussions. 
The authors thank the referee for many helpful comments and suggestions. 

 \end{document}